\documentclass{amsart}

\usepackage{comment}
\usepackage[breaklinks=True]{hyperref}
\usepackage{amsmath}
\usepackage{amssymb}
\usepackage[linesnumbered,ruled,vlined]{algorithm2e}
\usepackage{makecell}
\usepackage{fancyhdr}
\usepackage{tikz}
\usepackage{subfigure}
\usepackage[section]{placeins}

\newtheorem{theorem}{Theorem}[section]
\newtheorem{lemma}[theorem]{Lemma}

\newtheorem{definition}[theorem]{Definition}
\newtheorem{example}[theorem]{Example}
\newtheorem{xca}[theorem]{Exercise}
\newtheorem{corollary}[theorem]{Corollary}
\newtheorem{proposition}[theorem]{Proposition}
\newtheorem{remark}[theorem]{Remark}

\numberwithin{equation}{section}

\newcommand{\abs}[1]{\lvert#1\rvert}

\newcommand{\blankbox}[2]{%
  \parbox{\columnwidth}{\centering
    \setlength{\fboxsep}{0pt}%
    \fbox{\raisebox{0pt}[#2]{\hspace{#1}}}%
  }%
}

\begin{document}

\title[The $k$-regular partition function modulo composite integers $M$]{Distribution of the $k$-regular partition function modulo composite integers $M$}

\author{Yiwen Lu}
\address{Department of Mathematics, Nanjing University, Nanjing 210093, People’s Republic of China}
\email{luyw@smail.nju.edu.cn}
\author{Xuejun Guo}
\address{Department of Mathematics, Nanjing University, Nanjing 210093, People’s Republic of China}
\email{guoxj@nju.edu.cn}
\thanks{The authors are supported by NSFC 11971226 and NSFC 12231009.}




\begin{abstract}
	Let $b_k(n)$ denote the $k-$regular partitons of a natural number $n$. In this paper, we study the behavior of $b_k(n)$ modulo composite integers $M$ which are coprime to $6$. Specially, we prove that for arbitrary $k-$regular partiton function $b_k(n)$ and integer $M$ coprime to $6$,  there are infinitely many  Ramanujan-type congruences of $b_k(n)$ modulo $M$.
\end{abstract}

\maketitle

\section{\textbf{Introduction and statement of results}}
Suppose that  $n \geqslant 1$ is an integer, a partition of $n$ is any nonincreasing sequence of natural numbers, called parts, whose sum is $n$. Generally, denote by $p(n)$ the number of partitions for $n \geqslant 1$. For convenience, let $p(0) = 1$ and $p(n)=0 $ if $n$ is not a positive integer. Then when adding some special restrictions to the above partition, we will get a variety of functions with interesting arithmetic properties. Now we  introduce our main research object in this paper.

For an arbitrary integer $k\geqslant2$,  define a $k-$regular partition for a natural number $n$ to be a partition that none of its pairts have $k$ as a factor. Respectively, we denote by $b_k(n)$ the number of $k-$regular paititions of $n$. We specify $b_k(0)=1$ and $b_k(n)=0$ for $n\notin Z_{\geqslant0}$ similarly. Euler made lots of very significant contributions to the research of partitions, including his discovery of the generating functions for many types of partition functions. For $k$-regular partitions, we have 
\begin{equation}\nonumber
	\sum_{n=0}^{\infty}b_k(n)q^n=\prod_{n=1}^{\infty}\frac{1-q^{kn}}{1-q^n}.
\end{equation}
Let $m \geqslant 5$ be a prime number, and for an arbitrary integer $k\geqslant2$ we define the power series
\begin{equation}\label{FKM}
	F_{k,m}(z) := \sum_{\substack{n\geqslant 0 \\ mn\equiv k-1   \;(\mathrm{mod}\;24)}} b_k\left ( \frac{mn-k+1}{24}  \right ) q^n.
\end{equation}

\begin{theorem}\label{consturction}
	If $m\geqslant5$ is prime and $k\geqslant2$ is an  integer, then let $	F_{k,m}$ be defined as in $(\ref{FKM})$. For each positive integer $j$, there exists a modular form 
$$f_{k,m,j}\in S_{\frac{m^{j+1}-m^j}{2} }( \Gamma _0(576km), ( \frac{km}{\boldsymbol{\cdot} })) $$
with the property that 
$$f_{k,m,j}\equiv F_{k,m}\;(\mathrm{mod}\;m^{j+1} ).$$
\end{theorem}

\noindent Specific construction of the cusp forms $f_{k,m,j}$  in the above theorem will be described in the next section. 

For the case $j = 0$, the construction of the cusp forms are different from the ones in the above theorem. In fact, the following proposition shows us that $F_{k,m}$ is congruent modulo $m$ to a cusp form which is in the space $S_{\frac{m^{2}-m}{2} }\left( \Gamma _0(576km), ( \frac{km}{\boldsymbol{\cdot} })\right)$  when fixing $k$ and $m$.  

\begin{proposition}\label{j=0}
	If $m\geqslant5$ is prime and $k\geqslant2$ is an  integer, then there exists a modular form 
	$$f_{k,m} \in S_{\frac{m^{2}-m}{2} }( \Gamma _0(576km), ( \frac{km}{\boldsymbol{\cdot} })) $$
	with the property
	$$	f_{k,m} \equiv	F_{k,m}(z) \;(\mathrm{mod}\;m).$$
\end{proposition}

\noindent The forms $f_{k,m}$  will also be described in detail in the next section.

\begin{theorem}\label{boring}
	Let $m\geqslant5$ be prime, $k\geqslant2$ and $j$ be positive integers. Suppose there is one $n_0\equiv 24^{-1}(1-k) \;(\mathrm{mod}\;m)$  such that 
	$$b_k(n_0)\equiv e \;(\mathrm{mod}\;m^j),\quad  (e,m)=1.$$
	Then 
	$$\#\left \{ n<N:b_k(n)\equiv r\;(\mathrm{mod}\;m^j ) \right \} \gg \begin{cases}
		N & \text{ if } r\equiv 0\;(\mathrm{mod}\;m^j ) \\
		N/\log N   & \text{ if } r\not\equiv 0\;(\mathrm{mod}\;m^j ).
	\end{cases}$$
\end{theorem}

Moreover, we find that Theorem \ref{FKM} has provided conditions under which we can explore the arithmetic properties of congruences for an arbitrary $k-$regular function $b_k(n)$ with composite modulo $M$. The technique is to use the classical theories of Serre and the action of $Hecke$ $operators$ (similar to the operation of \cite{Ono} and \cite{Scott com M}).

\begin{theorem}\label{thm 1.3}
	Suppose $M$ is a positive integer which is coprime to $6$. Let $P_{M}$ be the product of all of the prime factors of $M$. Then for an abtrary integer $k\geqslant2$, there exists a positive density of the squarefree integers $L\equiv \pm 1 \;(\mathrm{mod}\;576kM)$ have the property that 
	$$ b_k\left ( \frac{P_{M}Ln-k+1}{24}  \right ) \equiv 0 \;(\mathrm{mod}\;M)$$
	for every positive integer $n$ coprime to $L$.
\end{theorem}

\begin{remark}
	By the above Theorem, one may find that for any positive integers $M$ coprime to $6$ there exists Ramanujan-type congruences modulo $M$.  In fact, it is only necessary to choose a special arithmetic progression modulo $24$ for $n$. For instance, in the proof of Theorem \ref{thm 1.3} in Section 3, one can control the arbitrary prime factor of $L$ coprime to $6$ by appropriate adjustments. Then replacing $n$ by $24nL+(k-1)P_{M}L+24$, we get $b_k ( P_ML^2n+P_ML+(k-1)\frac{{P_M}^2L^2-1}{24}  ) \equiv 0\; (\mathrm{mod}\;M).$ Naturally, we have the following corollary.
\end{remark}

\begin{corollary}
	If $M$ is a positive integer coprime to $6$ and $k\geqslant2$ is an arbitrary integer, then there exists infinitely many distinct arithmetic progressions $An + B$  with the property that 
	$$b_k(An+B)\equiv 0\;(\mathrm{mod}\;M)$$
	for every integer $n$.
\end{corollary}

As an immediate corollary of the above fact, we have

\begin{corollary}\label{Rough estimate}
	If $M$ is a positive integer coprime to $6$ and $k\geqslant2$ is an arbitrary integer, then there are infinitely many positive integers $n$ with the property that 
	$$b_k(n)\equiv 0 \;(\mathrm{mod}\;M).$$
	More specifically,
	$$\# \left \{ 0\leqslant n\leqslant N\::\: b_k(n)\equiv 0\;(\mathrm{mod}\;M ) \right \} \gg N.$$
\end{corollary}

\section{\textbf{Preliminaries}}
We expect that in this paper our proofs will be self-contained enough to achieve the goal of making it easier and more convenient for reading. Therefore, in the preliminary knowledge, we shall introduce background of modular forms sufficiently, including some theories abouot $Hecke$ $operator$, etc. To see more details, one can consult monographs \cite{GTM 97,GSM 179,R-Roy,GTM 228} or standard text \cite{milneMF,Notes}.

First we introduce some basic definitions for modular forms. The $upper$ $half$ $plane$, we denoting by $\mathcal{H}$, is one set:
$$\mathcal{H}=\{z \in \mathbb{C}, Im(\tau)>0\},$$
i.e., the complex numbers with positive imaginary pars.  Let $SL_2(\mathbb{Z})$ be the group of $2 \times 2$ integer matrices $\gamma = \begin{pmatrix}
	a &b \\
	c &d
\end{pmatrix}$
with $ab-cd=1$. For some $N\in \mathbb{Z}^{+}$, the congruence subgroups $\Gamma_0(N)$ is 
$$\Gamma_0(N):=\left \{ \begin{bmatrix}
	a&b \\
	c&d
\end{bmatrix}\in SL_2(\mathbb{Z}),\, \begin{bmatrix}
	a& b\\
	c &d
\end{bmatrix}\equiv \begin{bmatrix}
	* &* \\
	0 &*
\end{bmatrix}\;(\mathrm{mod} \;N )\right \}.$$

Let $k$ be an integer, then we denote by $M_k(\Gamma_0(N),\chi)$ (resp. $S_k(\Gamma_0(N),\chi)$) the usual space of holomorphic modular forms
(resp. cusp forms) of integral weight $k$ and Nebentypus
character $\chi$ with respect to the $\Gamma_0(N)$ where $\chi$ is a Dirchlet character modulo $N$.  In the case that $\lambda \in \frac{1}{2}\mathbb{Z}$ and $4\mid N$, we let  $M_\lambda(\Gamma_0(N),\chi)$ (resp. $S_\lambda(\Gamma_0(N),\chi)$) denote the corresponding holomorphic modular forms (resp. cusp forms) of half integer weight $\lambda$ similarly. Additionally, we let $\chi_d$ be the Kronecker character of $Q(\sqrt{d})$ for a squarefree integer $d$. 

Recall an instance of cusp form, the Dedekind's eta-function. It is defined by 
$$\eta(z):=q^{1/24}\prod_{n=1}^{\infty}(1-q^n)$$
 where $q=e^{2\pi i z}$. Thoughout this paper let  $q:=e^{2\pi i z}$. Furthermore, we know $\eta (24z)\in S_{\frac{1}{2} } (\Gamma (576), \chi _3)$. As one will see later in this paper, the above cusp forms will be used several times in the construction of modular forms we need. 

The $Hecke$ $operators$  which act on spaces of modular forms are natural linear transformations. We recall the action of $Hecke$ $operators$  on usual space of holomorphic modular forms of integral weight (This can be found in a number of standard texts such as \cite{Shi,Se,SLang}). 
\begin{definition}\label{def of hecke operator}
If $f(z)=\sum_{n=0}^{\infty } a(n)q^n\in M_k(\Gamma_0(N),\chi)$ and $p$ is a  prime, then the action of the $Hecke$ $operator$ $T_{p,k,\chi}$ on $f(z)$ is defined by
$$f(z)\mid T_{p,k,\chi} :=\sum_{n=0}^{\infty } (a(pn)+\chi(p)p^{k-1}a(n/p))q^n.$$
If $p\nmid n$, we let $a(n/p)=0$ for convince.
\end{definition}
\noindent For a specific space of modular forms $M_k(\Gamma_0(N),\chi)$, we usually abbreviate $T_{p,k,\chi}$ as $T_p$. Further, $T_p$ is a $\mathbb{C}-$vector space endomorphism. 

In this paper, for a prime $l$, we let $S_t(\Gamma_0(N),\chi)_l$  be the $\mathbb{F}_l$-vector space of reductions mod $l$ of the $q$-expansions of  forms in $S_t(\Gamma_0(N),\chi)$ with rational integer coefficients, here $t \in \mathbb{Z}$ or $\frac{1}{2}\mathbb{Z}$.  Then we present a Serre's powerful result with respect to $Hecke$ $operators$ (c.f. \cite[6.4]{Se operator}).
\begin{theorem}[J.-P. Serre]
	\label{Serre's theorem}
	The set of primes $l\equiv-1\;(\mathrm{mod}\;N)$ such that 
	$$f\ |\ T(l)\equiv0\; (\mathrm{mod}\; m)$$
	
	\noindent
	for each $f(z)\in S_k(\Gamma_0(N),\chi)_m$ has positive density, here $T(l)$ denotes the usual Hecke operator of index l acting on $S_k(\Gamma_0(N),\chi)$.
\end{theorem}
The above result take a critical role in the proof of Thorem \ref*{thm 1.3}.  Next we introduce two other operators which act on the formal power series $\sum_{n\geqslant 0}a(n)q^n.$

Let $d_1$ and $d_2$ be two integers. Then define 
\begin{equation}\label{U operator}
	\sum_{n\geqslant 0}a(n)q^n  \mid U(d_1)=\sum_{n\geqslant 0}a(d_1n)q^n,
\end{equation}

\begin{equation}\label{V operator}
	\sum_{n\geqslant 0}a(n)q^n \mid V(d_2) =\sum_{n\geqslant 0}a(n)q^{{d_2}n} .
\end{equation}
Obviously, the operator $U(d_1)$ acts on series expansions also by
\begin{equation}\label{U OPERATOR}
	\sum_{n\geqslant 0}a(n)q^n \mid  U(d_1)=\sum_{\substack{n\geqslant 0 \\ n\equiv 0\;(\mathrm{mod}\;p) } }a(n)q^{\frac{n}{d_1} }.
\end{equation}
Additionally, the following lemma describes the behavior of these operators on half-integral weight modular forms (see \cite{Web UV operator }, Proposition 3.7).

\begin{lemma}\label{U V lemma}
	Suppose that $f(z)\in M_{\lambda + \frac{1}{2}} (\Gamma _0(4N),\chi ).$
	\begin{enumerate}
		\item If $d$ is a positive integer, then $f(z) \mid V(d) \in  M_{\lambda + \frac{1}{2}} (\Gamma _0(4Nd),(\frac{4d}{\boldsymbol{\cdot}}) \chi)$. Moreover, if $f(z)$ is a cusp form, then so is $f(z) \mid V (d)$.
		\item If $d\mid N$, then $f(z) \mid U(d) \in  M_{\lambda + \frac{1}{2}} (\Gamma _0(4N),(\frac{4d}{\boldsymbol{\cdot}}) \chi ).$ Moreover, if $f(z)$ is a cusp form, then so is $f(z) \mid U (d)$.
	\end{enumerate}
\end{lemma}

\begin{remark} 
	Lemma \ref{U V lemma} (2) applies to all forms $f(z)$. Specially, if $d\nmid N$, then simply view $f(z)$ as an element of $M_{\lambda + \frac{1}{2}} (\Gamma _0(4dN),\chi ).$
\end{remark}

\noindent For an arbitrary integer $k\geqslant2$ and a prime number $m\geqslant5$,  define
\begin{equation}\label{def of fkm}
	f_{k,m} := \frac{\{\eta (24kz)\eta ^{m^2-1}(24z)\}\mid U(m)}{\eta^m (24z)}.
\end{equation}
We will show in the next section that the functions defined above  satisfy the conditions of Proposition \ref{j=0}.

Let  $j$ be a positive integer, $m\geqslant5$ be prime and $k\geqslant2$ be an arbitrary integer. Then we denote $f_{k,m,j} $ by 
\begin{equation}\label{def of FKMJ}
	f_{k,m,j} := \left \{ \frac{\eta(24kz) \eta^m(24mz) }{\eta(24z)} \right \}\mid U(m)\cdot \frac{\eta^{m^{j+1}-m}(24z)}{\eta^{m^j}(24mz)}.  
\end{equation}

\begin{proposition}\label{prop FKMJ congruence}
	Let $j$ be a positive integer, and let $m\geqslant5$ be prime. Then for an aibitrary integer $k\geqslant2$
		$$	f_{k,m,j} \equiv	F_{k,m}(z) \;(\mathrm{mod}\;m^{j+1}).$$
\end{proposition}

\begin{proposition}\label{prop FKMJ cusp}
	Let $j$ be a positive integer, and let $m\geqslant5$ be prime. Then for an aibitrary integer $k\geqslant2$
	$$f_{k,m,j} \in S_{\frac{m^{j+1}-m^j}{2} }( \Gamma _0(576km), ( \frac{km}{\boldsymbol{\cdot} }))$$
\end{proposition}
\noindent It is obvious that Theorem \ref{consturction} immediately follows from the two propositions above.

\section{\textbf{Proof of the results}}
Before formally starting our proof of the main results, we proceed with the proof of two Lemmas. They will be used for the construction of the modular form and for the illustration of the congruence properties in this paper.

\begin{lemma}\label{communicative}
	If $d_1$ and $d_2$ are two positive integers of coprime, then for a formal power series $\sum_{n\geqslant 0}a(n)q^n$ we have 
	$$\left \{ \sum_{n\geqslant 0}a(n)q^n \mid U(d_1) \right \} \mid V(d_2)=\left \{ \sum_{n\geqslant 0}a(n)q^n \mid V(d_2) \right \} \mid U(d_1)$$
	where $U$ and $V$ are the operators defined as in $(\ref{U operator})$ and $(\ref{V operator})$ respectively. 
\end{lemma}

\begin{proof}We deduce from \ref{U OPERATOR} and \ref{V operator} that 
   $$\left \{ \sum_{n\geqslant 0}a(n)q^n \mid U(d_1) \right \} \mid V(d_2)= \sum_{\substack{n\geqslant 0 \\ n\equiv 0\;(\mathrm{mod}\;d_1) } }a(n)q^{\frac{n d_2}{d_1  }}=\sum_{\substack{n\geqslant 0 \\ nd_2\equiv 0\;(\mathrm{mod}\;d_1) } }a(n)q^{\frac{n d_2}{d_1  }}$$
   
   $$=\left \{ \sum_{n\geqslant 0}a(n)q^n \mid V(d_2) \right \} \mid U(d_1).$$
\end{proof}

\begin{lemma}\label{congruence}
	Suppose that $n_1, n_2$ are two positive integers with ${n_2 \geqslant2}$. Then for each positive integer $i$ we have,
	$$\frac{\eta^{n^i_2}(n_1z)}{\eta^{n^{i-1}_2}(n_1n_2z)} \equiv 1\;(\mathrm{mod} \;n^i_2).$$
\end{lemma}

\begin{proof}
 Since $(1-X^{n_1})^{n_2}=(1-X)^{n_1 n_2}\;(\mathrm{mod}\;n_2)$, it follows immediately that 
 $$\frac{\eta^{n_2}(n_1z)}{\eta(n_1n_2z)}=\prod_{n=1}^{\infty } \frac{(1-q^{24n_1})^{n_2}}{(1-q^{24n_1n_2})}  \equiv 1\;(\mathrm{mod} \;n_2).$$
 Assume that the congruence relation in the Lemma holds in the case $i = k - 1, k\geqslant2$ is an integer, i.e., 
 $$\frac{\eta^{n^{k-1}_2}(n_1z)}{\eta^{n^{k-2}_2}(n_1n_2z)} =1+n^{k-1}_2g(z),$$
 where $g(z)$ is a power series with all integers as coefficients. Therefore 
  $$\frac{\eta^{n^{k}_2}(n_1z)}{\eta^{n^{k-1}_2}(n_1n_2z)}=[1+n^{k-1}_2g(z)]^{n_2}=1+\sum_{j=1}^{{n_2}-1}\binom{n_2}{j} [n^{k-1}_2g(z)]^j+[n^{k-1}_2g(z)]^{n_2}.$$
  Note that $n_2 \mid \binom{n_2}{j}   $ with $1 \leqslant j \leqslant {n_2} -1$, and $(k-1)n_2 \geqslant k$. Then it follows from Mathematical Induction that Lemma holds.
\end{proof}

\begin{proof}[Proof of Proposition \ref{prop FKMJ congruence}]
First recall  the definition of $F_{k,m}(z)$ in (\ref{FKM}) and $f_{k,m,j}(z)$ in (\ref{def of FKMJ}).  Combining Lemma \ref{congruence} we find that it is sufficient for us to verify that
\begin{equation}\label{suffice to prove}
	\left \{ \frac{\eta(24kz) \eta^m(24mz) }{\eta(24z)} \right \}\mid U(m)= \sum_{\substack{n\geqslant 0 \\ mn\equiv k-1   \;(\mathrm{mod}\;24)}} b_k\left ( \frac{mn-k+1}{24}  \right ) q^n \cdot \eta^m(24z)
\end{equation}

Note that the generating functions for $b_k(n)$ are 
$$	\sum_{n=0}^{\infty}b_k(n)q^n=\prod_{n=1}^{\infty}\frac{1-q^{kn}}{1-q^n},$$
therefore 
\begin{equation}\nonumber
	\begin{aligned}
		\frac{\eta(24kz) \eta^m(24mz) }{\eta(24z)} &= q^{k-1}\frac{\prod_{n=1}^{\infty}(1-q^{24kn})}{\prod_{n=1}^{\infty}(1-q^{24n})}\cdot q^{m^2}\cdot \prod_{n=1}^{\infty}(1-q^{24mn})^m \\
		&=\sum_{n=0}^{\infty}b_k(n)q^{24n+k-1}\cdot q^{m^2}\cdot \prod_{n=1}^{\infty}(1-q^{24mn})^m.
	\end{aligned}
\end{equation}
Thus the action of $U$ operator (\ref{U OPERATOR}) shows that

\begin{equation}\nonumber
	\begin{aligned}
	\left \{ \frac{\eta(24kz) \eta^m(24mz) }{\eta(24z)} \right \}\mid U(m)&= \sum_{\substack{n\geqslant 0\\24n\equiv 1-k\;(\mathrm{mod}\;m )} }^{\infty }b_k(n) q^{\frac{24n+k-1}{m} }\cdot q^m \cdot \prod_{n=1}^{\infty}(1-q^{24n})^m \\
	&=\sum_{\substack{n\geqslant 0\\24n\equiv 1-k\;(\mathrm{mod}\;m )} }^{\infty }b_k(n) q^{\frac{24n+k-1}{m}}\cdot\eta ^m(24z).
		\end{aligned}
\end{equation}
Using the fact 
$$\sum_{\substack{n\geqslant 0\\24n\equiv 1-k\;(\mathrm{mod}\;m )} }^{\infty }b_k(n) q^{\frac{24n+k-1}{m}}=\sum_{\substack{n\geqslant 0\\mn\equiv k-1\;(\mathrm{mod}\;24 )} }^{\infty }b_k(\frac{mn-k
	+1}{24} ) q^n,$$
we complete the verification of (\ref*{suffice to prove}), meaning  the end of the proof of Proposition \ref*{prop FKMJ congruence}.
\end{proof}

\begin{proof}[Proof of Proposition \ref{prop FKMJ cusp}]
First we shall illustrate that which space of modular forms $\eta^m(24kz)\mid U(m)$ belongs to. 

Recall that $\eta (24z)\in S_{\frac{1}{2} } (\Gamma (576), \chi _3)$. Therefore by Lemma \ref{U V lemma}, we know that 
$$\eta(24kz)=\eta(24z) \mid V(k) \in  S_{\frac{1}{2} } (\Gamma (576k), (\frac{k}{\boldsymbol\cdot } )\chi _3),$$
and 
\begin{equation}\label{eta in which modular forms}
	\eta(24kz)\mid U(m) \in  S_{\frac{1}{2} } (\Gamma (576km), (\frac{km}{\boldsymbol\cdot } )\chi _3)
\end{equation}

We deduce from the definition \ref{def of FKMJ} and Lemma \ref{communicative} that  
\begin{equation}\nonumber
	\begin{aligned}
	f_{k,m,j} &= \left \{ \frac{\eta(24kz) \eta^m(24mz) }{\eta(24z)} \right \}\mid U(m)\cdot \frac{\eta^{m^{j+1}-m}(24z)}{\eta^{m^j}(24mz)}\\
	&=\eta(24kz)\mid U(m)\cdot \left \{ \frac{\eta^m(mz)}{\eta(z)}\mid V(24)  \right \} \mid  U(m)\cdot \frac{\eta^{m^{j+1}-m}(24z)}{\eta^{m^j}(24mz)}\\	
	&=\eta(24kz)\mid U(m)\cdot \left \{ \frac{\eta^m(mz)}{\eta(z)}\mid U(m)  \right \} \mid  V(24)\cdot \frac{\eta^{m^{j+1}-m}(24z)}{\eta^{m^j}(24mz)}.
	\end{aligned}	
\end{equation}
By a fact on power series \cite[Proposition 2]{Scott com M}, we see that 
\begin{equation}\label{shenlaizhibi}
	\left \{ \frac{\eta^m(mz)}{\eta(z)}\mid U(m)  \right \} \mid  V(24)\cdot \frac{\eta^{m^{j+1}-m}(24z)}{\eta^{m^j}(24mz)}\in S_{\frac{m^{j+1}-m^j-1}{2} }( \Gamma _0(576m), \chi _3).
\end{equation}
Putting (\ref{eta in which modular forms}) and (\ref{shenlaizhibi}) together, one completes the proof of Proposition \ref{congruence}.
\end{proof}

Next we proceed to the proof of Proposition \ref*{j=0}, that is the case when Theorem \ref*{consturction} holds for $j=0$. The proof uses a technique similar to that of Theorem \ref*{consturction}, as follows

\begin{proof}[Proof of Proposition \ref{j=0}]
 We start by describing the space of modular forms that $f_{k,m}(z)$ belongs to.
 
 \noindent From the definition of $f_{k,m}$ (\ref{def of fkm}) and Lemma \ref{U V lemma}, we know
 \begin{equation}\nonumber
 	\begin{aligned}
 		f_{k,m}&= \eta (24kz) \mid U(m)\cdot\frac{\eta ^{m^2-1}(24z)\mid U(m)}{\eta^m (24z)}  \\
 		&=\eta (24kz) \mid U(m)\cdot\frac{(\eta ^{m^2-1}(z))\mid U(m)\mid V(24)}{\eta^m (24z)}.
 	\end{aligned}
 \end{equation}
One can see in \cite[Theorem 8]{Ono} that $$\eta (24kz) \mid U(m)\cdot\frac{(\eta ^{m^2-1}(z))\mid U(m)\mid V(24)}{\eta^m (24z)} \in S_{\frac{m^2-m-1}{2} }(\Gamma _0(576m),\chi _3).$$
Note that we have illustrated that $\eta(24kz)\mid U(m) \in  S_{\frac{1}{2} } (\Gamma (576km), (\frac{km}{\boldsymbol\cdot } )\chi _3)$ in the proof of Proposition \ref*{prop FKMJ congruence}, therefore combining the above fact we get
$$f_{k,m} \in S_{\frac{m^{2}-m}{2} }( \Gamma _0(576km), ( \frac{km}{\boldsymbol{\cdot} })).$$

Next we proceed to prove the congruence relation for $f_{k,m}$.

\noindent It is easy to see that
\begin{equation}\nonumber
	\begin{aligned}
		 \frac{\eta (24kz)\eta ^{m}(24mz)}{\eta (24z)} &=q^{k-1}\cdot \frac{\prod_{n=1}^{\infty } (1-q^{24kn})}{\prod_{n=1}^{\infty } (1-q^{24n})}\cdot q^{m^2}\cdot \prod_{n=1}^{\infty }(1-q^{24mn})^{m}  \\
		&=\sum_{n=0}^{\infty }b_k(n)q^{24n+k-1}\cdot q^{m^2}\cdot \prod_{n=1}^{\infty }(1-q^{24mn})^{m}.
	\end{aligned}
\end{equation}
Thus we get 
\begin{equation}\nonumber
\begin{aligned}
	 \frac{\eta (24kz)\eta ^{m}(24mz)}{\eta (24z)} \mid U(m)&=\sum_{\substack{n\geqslant0\\24n \equiv 1-k\;(\mathrm{mod}\;m )} }^{\infty }b_k(n)q^{\frac{24n+k-1}{m} }\eta ^{m}(24z) \\
	 &=\sum_{\substack{n\geqslant 0 \\ mn\equiv k-1   \;(\mathrm{mod}\;24)}} b_k\left ( \frac{mn-k+1}{24}  \right ) q^n\cdot \eta ^{m}(24z).
\end{aligned}
\end{equation}\label{a(n)}
Since $(1-X^m)^m\equiv (1-X)^{m^2}\;(\mathrm{mod}\;m )$, we find that $$\frac{\{\eta (24kz)\eta ^{m^2-1}(24z)\}\mid U(m)}{\eta^m (24z)}\equiv \sum_{\substack{n\geqslant 0 \\ mn\equiv k-1   \;(\mathrm{mod}\;24)}} b_k\left ( \frac{mn-k+1}{24}  \right ) q^n \;(\mathrm{mod}\; m).$$
\end{proof}

Next we proceed to the proof of Theorem \ref*{boring} which follows \cite{lovejoy} . First we introduce the following theorem due to Serre \cite{Se operator}. 
\begin{theorem}[Serre]\label{sea(n)}
	Supose that $F(z):=\sum_{n=1}^{\infty}a(n)q^n$ is an integer weight cusp form with coefficients in $\mathbb{Z}$. If $M$ is a positive integer, then there exists a set of primes $S_M$  of positive density with the property that 
	$$a(nl^r)\equiv (r+1)a(n) \;(\mathrm{mod}\;M)$$
	whenever $l \in S_M$, $r$ is a positive integer, and $n$ is coprime to $l$. 
\end{theorem}

\begin{proof}[Proof of Theorem \ref{boring}]
	For the fixed positive integers $m^j$ and $k$, by Theorem \ref{consturction}, there exists a cusp form $f_{k,m,j}$  with rational coefficients. Denote by $S_{m^j}$ the set of primes for $f_{k,m,j}$  guaranteed by the above theorem. The quantitative estimate in the case $r\equiv 0 \;(\mathrm{mod}\;m^j)$ follows from Corollary \ref{Rough estimate}. 
	
	For the case $r\not\equiv 0 \;(\mathrm{mod}\;m^j)$, we recall that there is one $n_0=\frac{mn-k+1}{24}$ $(n\in \mathbb{Z})$ such that 
	$$b_k(n_0)\equiv e \;(\mathrm{mod}\;m^j),\quad (e,m)=1.$$
    Then for each $i=1,2,\cdots,m^j-1$, let $r_i \equiv i(2e)^{-1}-1 \;(\mathrm{mod}\;m^j)$ and $r_i>0$. Obviously, there exists a prime $l_0\in S_{m^j}$ coprime to $n$ that
     $$b_k \left (  \frac{mnl_0^{r_i}-k+1}{24}  \right ) \equiv (r_i +1)b_k(n_0) \;(\mathrm{mod}\;m^j)$$
    for each $i$ $(1\leqslant i \leqslant m^j-1)$. Fix $l_0$, then for a given $i$, use Theorem \ref{sea(n)} again, we find for all but finitely primes $l \in S_{m^j}$,
$$b_k \left (  \frac{mnl_0^{r_i}l-k+1}{24}  \right ) \equiv 2(r_i +1)b_k(n_0) \equiv  i \;(\mathrm{mod}\;m^j)$$
for each $i$.
    Since the set of primes $S_{m^j}$  has a positive density, we finally get 
    $$\#\left \{ n<N:b_k(n)\equiv r_i\;(\mathrm{mod}\;m^j ) \right \} \gg \pi(N) \gg \frac{N}{\log N }.$$

\end{proof}

\begin{remark}
	In fact, one can have a better result for Theorem \ref*{boring}. For instance, for a fixed integer $k\geqslant2$ and a prime $m\geqslant5$, one can obtain a value range of the $n_0$ in theorem if it exists. The technique is following Lovejoy \cite{lovejoy} that use the Sturm's bound \cite[Theorem 1]{sturm bound} as follows. 
\end{remark}

\begin{theorem}[J. Sturm]
	\label{Sturm's theorem}
	Suppose $f(z)=\sum_{n=0}^\infty a(n)q^n\in M_k(\Gamma_0(N),\chi)_m$ such that
	$$a(n)\equiv0\ (\mathrm{mod}\ m)$$
	
	\noindent
	for all $n\leq \frac{kN}{12}\prod_{p|N}\left( 1+\frac1p \right)$. Then $a(n)\equiv0\ (\mathrm{mod}\ m)$ for all $n\in\mathbb{Z}$.
\end{theorem}

\begin{proof}[Proof of Theorem \ref{thm 1.3}]
Suppose that $M$ is a positive integer which is coprime to $6$. Let $M=\prod_{i=1}^{I}m_i^{j_i}$ be its prime factorization. Then by definition $P_M$ is defined that $P_M :=\prod_{i=1}^{I}m_i.$ By Theorem \ref{consturction}, for each $i$ and integer $k\geqslant2$ we have corresponding cusp forms as follows 
\begin{equation}\label{fact_1}
	f^{k,m_i}(z)\in S_{\frac{m^{i+1}-m^i}{2} }( \Gamma _0(576km_i), ( \frac{km_i}{\boldsymbol{\cdot} })) 
\end{equation}
with the property that 
\begin{equation}\label{fact_2}
	f^{k,m_i}(z)\equiv \sum_{\substack{n\geqslant 0 \\ m_in\equiv k-1   \;(\mathrm{mod}\;24)}} b_k\left ( \frac{m_in-k+1}{24}  \right ) q^n \;(\mathrm{mod}\; m_i^{j_i}).
\end{equation}

Our next  argument mainly applies  Serre's important results, and this discussion is similar to \cite{Ono,Scott com M}.
Note that each form $f^{k,m_i}(z)$ has rational integer coefficients, then we can use Theorem \ref*{Serre's theorem} for all forms above.

For every $m_i$ we write $f^{k,m_i}(z)=\sum_{n=1}^{\infty } a^{k,m_i}(n)q^n$, then Theorem \ref{Serre's theorem} tells us for each form $f^{k,m_i}(z)$ there exists a corresponding positive density set of the primes $S_i$ have the property
\begin{equation}\nonumber
	f^{k,m_i}(z) \mid T(l_i^j) \equiv 0 \;(\mathrm{mod}\; m_i^{j_i})
\end{equation}
where $T(l_i^j)$ denote the usuall $Hecke$ $operator $ for an arbitrary prime $l_i^j \in S_i$. Morever, $ l_i^j  \equiv -1 \;(\mathrm{mod}\; 576kM)$. Then by Definition \ref{def of hecke operator} and the fact (\ref{fact_1}), we can see 
$$a^{k,m_i}(l_i^{j}n)+(\frac{km_i}{l_i^j})\cdot(l_i^j)^{{\frac{m^{i+1}-m^i-2}{2} }}\cdot a^{k,m_i}(n/l_i^j)\equiv 0\;(\mathrm{mod}\;m_i^{j_i})$$
for all primes $l_i^j \in S_i$. Recall from (\ref{fact_2}) that $a^{k,m_i}(n)\equiv b_k\left ( \frac{m_in-k+1}{24}  \right ) \;(\mathrm{mod}\;m_i^{j_i})$ vanishs when $\frac{m_in-k+1}{24} $ in not an integer, then we get 
\begin{equation}\label{Key}
	a^{k,m_i}(l_i^{j}n)=b_k\left ( \frac{m_il_i^jn-k+1}{24}  \right )\equiv 0\;(\mathrm{mod}\;m_i^{j_i})
\end{equation}
for all primes $l_i^j \in S_i$ and every positive integer $n$ coprime to $l_i^j$.

Next, for each set of prime numbers $S_i$ $(1\leqslant i \leqslant I-1)$, we fix a corresponding prime number $l_i$ $(1\leqslant i \leqslant I-1)$ respectively. And since $S_i$ are all positive density sets, we can guarantee that the selected prime numbers  $l_i$ are pairwise coprime and all coprime to $P_M$ . As for the set $S_I$, we choose all the prime numbers that are different from $l_i$ $(1\leqslant i \leqslant I-1)$  and coprime to $P_M$ to form a new set which is denoted as $S_I^{\prime}$. 

Then we define a new set as follows
$$S=\left \{ L:L=l_{I^{\prime }}^j\cdot  \prod_{i=1}^{I-1} l_i,\ l_{I^{\prime }}^j\in S_I^{\prime} \right \}.$$
It is easy to see that $L\equiv \pm 1 \;(\mathrm{mod}\;576kM)$  and the two sets $S,$ $S_I$ have the same density. Morever, by (\ref{Key}), we finally get that 
$$ b_k\left ( \frac{P_{M}Ln-k+1}{24}  \right ) \equiv 0 \;(\mathrm{mod}\;M)$$
for every positive integer $n$ coprime to $L$.

This completes the proof.
\end{proof}

\end{document}